\documentclass[10pt, reqno]{amsart}
\usepackage{amssymb, amsmath, amsthm, color, hyperref, graphicx, hyperref}

\numberwithin{equation}{section}

\usepackage{cite}

\newcommand{\R}{\mathbb{R}}
\newcommand{\C}{\mathbb{C}}
\newcommand{\F}{\mathcal{F}}

\newcommand{\qtq}[1]{\quad\text{#1}\quad}

\newtheorem{theorem}{Theorem}[section]

\newtheorem{proposition}[theorem]{Proposition}

\theoremstyle{definition}

\newcommand{\eps}{\varepsilon}

\renewcommand{\Re}{\text{Re\,}}

\newcommand{\Expect}{{\rm I\kern-.3em E}}

\begin{document}

\title[Recovery for Cubic NLS]{Recovery of a Cubic Nonlinearity for the Nonlinear Schr\"odinger Equation}

\author[C. Hogan]{Christopher C. Hogan}
\email{cch62d@mst.edu}
\author[J. Murphy]{Jason Murphy}
\email{jason.murphy@mst.edu}
\author[D. Grow]{David Grow}
\email{grow@mst.edu}

\address{Department of Mathematics \& Statistics, Missouri S\&T}

\begin{abstract}
We consider the problem of recovering a spatially-localized cubic nonlinearity in a nonlinear Schr\"odinger equation in dimensions two and three.  We prove that solutions with data given by small-amplitude wave packets accrue a nonlinear phase that determines the X-ray transform of the nonlinear coefficient. 
\end{abstract}

\maketitle

\section{Introduction}

We consider the following cubic nonlinear Schr\"odinger equation (NLS): 
\begin{equation}\label{NLS}
\begin{cases}
(i\partial_t + \tfrac12\Delta)u = \alpha |u|^2 u, \\ u|_{t=-T} = u_0. \end{cases}
\end{equation}
Here $T>0$, $u:\R_t\times\R_x^d\to\C$ with $d\in\{2,3\}$, and $\alpha\in C_c^\infty(\R^d)$ is nonnegative.

We study the problem of recovering the nonlinear coefficient $\alpha$. Watanabe \cite{Watanabe4} previously studied a similar problem (under a slightly more general formulation) from the inverse scattering perspective and proved that knowledge of the scattering map suffices to reconstruct the nonlinearity.  In this work, we adapt the approach of S\'a Barreto and Stefanov \cite{SS21, SS21p2}, who considered a similar recovery problem for the cubic wave equation (as well as more general nonlinear wave models).  In particular, we show that solutions to \eqref{NLS} with data given by small-amplitude wave packets accrue a nonlinear phase that determines the X-ray transform of $\alpha$. 

Our main result is the following theorem.

\begin{theorem}\label{T} Let $d\in\{2,3\}$, $p>1$, $a_0,\alpha\in\mathcal{S}(\R^d)$, with $\alpha$ compactly supported.  Fix $T>0$ large enough that the support of $\alpha$ is contained in $\{|x|<T\}$. 

Given $\xi\in\R^d$ with $|\xi|=1$, define
\begin{equation}\label{initial-data}
u_0(x) = \eps^{p}a_0(\eps \left[x+T\xi\right]) \exp\{i\left[x\cdot\xi+\tfrac{1}{2}T\right]\}
\end{equation}
and
\begin{equation}\label{tildeu}
\begin{aligned}v(t,x) = &\eps^{p}a_0(\eps(x-t\xi)) \exp\{i\bigl[x\cdot\xi - \tfrac{1}{2}t\bigr]\} \\
&\quad\times \exp\biggl\{-i|\eps^p a_0(\eps(x-t\xi))|^2\int_{-T}^t \alpha(x -(t-s)\xi))\,ds\biggr\}.
\end{aligned}
\end{equation}

For $\eps=\eps(T,p,\|\hat\alpha\|_{L^1})>0$ sufficiently small, the solution $u$ to \eqref{NLS} exists on the time interval $[-T,T]$ and satisfies
\begin{equation}\label{bd}
\|u-v\|_{L_{t,x}^\infty([-T,T] \times \R^d)} \lesssim_\sigma \max\{\eps^{p+2},\eps^{3p-2-\sigma}\}
\end{equation}
for sufficiently small $\sigma>0$.
\end{theorem}

Theorem~\ref{T} shows that by solving \eqref{NLS} with data as in \eqref{initial-data}, we can effectively recover the X-ray transform of $\alpha$, which in turn suffices to reconstruct $\alpha$ itself (see e.g. \cite{Deans}).  Indeed, we first note the approximate solution $v(T,x)$ exhibits the integral of $\alpha$ over the line joining $x$ to $x-2T\xi$, so that by varying $\xi\in\mathbb{S}^{d-1}$ and $x\in\text{supp}(\alpha)\subset \{|x|<T\}$, we can obtain all line integrals of $\alpha$.  We then observe that in the regime $p>1$ and $\eps\ll 1$, we have that the error in \eqref{bd} is much smaller than the size of the solution and approximate solution themselves.  The proof will show that the implicit constant in \eqref{bd} will involve some positive power of $T$; however, as $\eps$ must already be chosen depending on $T$, this does not affect the main conclusion of the theorem.  Our proof will also require that $\eps$ be chosen sufficiently small depending on $\|\hat\alpha\|_{L^1}$, which means that some degree of \emph{a priori} information about $\alpha$ is necessary.  In practice, a bound for this norm could be obtained if one had estimates for $T$ and $\|\alpha\|_{W^{2,\infty}}$ (see e.g. \eqref{practical} below).

We were motivated to study this problem by the recent work \cite{SS21}, which proved an analogous result in the setting of the nonlinear wave equation.  In that work, the authors considered initial data of the form
\[
u(t,x) = h^{-\frac12} e^{i(x\cdot\xi-t)/h}\chi(x\cdot\xi-t)\quad (0<h\ll 1) 
\]
outside of the support of $\alpha$ and constructed a suitable geometric optics approximation involving the same nonlinear phase shift as the one appearing in Theorem~\ref{T}.  This work was further generalized in \cite{SS21p2}. Our construction of the approximate solution \eqref{tildeu} is a modification of the one appearing in the work \cite{CDS10}, which constructed geometric optics solutions in the NLS setting.  As in that paper, we will establish well-posedness and stability in the Wiener algebra $\F L^1$.  In fact, the proof of Theorem~\ref{T} will show that we actually control the error in \eqref{bd} in the stronger norm $L_t^\infty \F L^1$. 

As mentioned above, Theorem~\ref{T} is closely related to the work \cite{Watanabe4}, which considered the problem of reconstructing nonlinearities of the form $q(x)|u|^{p-1}u$ (with both $q$ and $p$ unknown) from knowledge of the scattering map (we refer the reader to \cite{CarGal, MorStr, SBUW, Weder6, Weder3, Sasaki3, KMV} for a sample of other results of this type).  In \cite{Watanabe4}, the coefficient $q$ was required to satisfy some smoothness and decay assumptions, but was not required to be compactly supported.  Our smoothness requirements on $\alpha$ are essentially the same as in \cite{Watanabe4} (the proof will show that we need to control at least $d/2$ derivatives).  However, the compact support assumption is essential for our arguments; indeed, the proof will show that $\eps\to 0$ if $T\to\infty$.  In this work, we regard the power of the nonlinearity as fixed.  For simplicity, we have chosen to work with the cubic power.  The same proof would treat more general nonlinearities of the form $\alpha(x) |u|^{2k} u$ for $k\in\mathbb{N}$, although our use of the space $\F L^1$ does require an algebraic nonlinearity (cf. \eqref{algebraic} below). 

Let us now briefly describe the strategy of the proof of Theorem~\ref{T}.  The key step is the construction of the approximate solution $v$ appearing in \eqref{tildeu}.  Keeping in mind the scaling symmetry and Galilean invariance of the underlying linear equation, we look for a solution of the form
\[
v(t,x) = \eps^p a(\eps^2[t+T],\eps[x+T\xi])\, e^{i[x\cdot\xi-\frac12|\xi|^2 t]},
\]
with $v|_{t=0}=u_0$ as in \eqref{initial-data}.  Direct computation shows that if we define $a$ to satisfy a suitable nonlinear transport equation, then $v$ solves \eqref{NLS} up to an error term involving only $\Delta a$.  The rest of the proof then consists of (i) establishing a suitable stability theory for \eqref{NLS} in $\F L^1$ and (ii) proving suitable estimates for the error term in this norm.  The use of the space $\F L^1$ was inspired by the paper \cite{CDS10} and turns out to be a convenient space for establishing well-posedness and stability.  Ultimately, we are able to obtain a suitable estimate in this space only in the regime $p>1$ (recall that the amplitude of the solution is $\eps^p$).  It is an interesting problem to consider the same problem formulated in more familiar spaces, such as the energy space $H^1$.

The rest of this paper is organized as follows: In Section~\ref{S:prelim} we introduce some notation and collect some basic estimates.  In Section~\ref{S:WP}, we prove well-posedness and stability of \eqref{NLS} in the space $\F L^1$.  In Section~\ref{S:approx}, we construct the approximate solution $v$ to \eqref{NLS} appearing in Theorem~\ref{T} and prove estimates for the error term involving $\Delta a$.  Finally, in Section~\ref{S:proof}, we carry out the proof of Theorem~\ref{T}.

\subsection*{Acknowledgements} J. M. was supported by a Simons Collaboration Grant.

\section{Preliminaries}\label{S:prelim} We write $A\lesssim B$ to denote $A\leq CB$ for some $C>0$.  Dependence on parameters will be indicated by subscripts, i.e. $A\lesssim_T B$ denotes $A\leq CB$ for some $C=C(T)>0$.  We use $C$ to denote a positive constant whose value may change from line to line. 

We make use of the Wiener algebra $\F L^1(\R^d)$, equipped with the norm
\[
\|u\|_{\F L^1} = \| \hat u\|_{L^1}.
\] 
Here $\hat u$ denotes the Fourier transform 
\[
\hat u(y)=\F u(y) = \int_{\R^d} e^{-ix\cdot y}u(x)\,dx.
\]
We record the following algebra property:
\begin{align*}
\| uv\|_{\F L^1} = \| \F[uv]\|_{L^1} = \|\hat u \ast \hat v\|_{L^1} \leq \|\hat u\|_{L^1} \|\hat v\|_{L^1} = \|u\|_{\F L^1}\|v\|_{\F L^1} 
\end{align*}
We also observe the embedding $\F L^1 \hookrightarrow L^\infty$, which follows from the Hausdorff--Young inequality: 
\[
\|u\|_{L^\infty}=\|\F^{-1}\hat u\|_{L^\infty} \lesssim \|\hat u\|_{L^1}=\|u\|_{\F L^1}. 
\]

Finally, we note that by Cauchy--Schwarz and Plancherel, we can obtain the following estimate:
\begin{equation}\label{practical}
\|u\|_{\F L^1} \lesssim_\sigma \|(1+|y|^2)^{\frac{d}{4}+\sigma}\hat u\|_{L^2} \lesssim_\sigma \|u\|_{H^{\frac{d}{2}+\sigma}} 
\end{equation}
for any $\sigma>0$. 

In what follows, we will write 
\begin{equation}\label{propagator}
e^{it\Delta/2}=\F^{-1}e^{-it|y|^2/2}\F
\end{equation}
for the free Schr\"odinger propagator.

\section{Well-posedness and stability}\label{S:WP}

In this section we establish well-posedness and stability for \eqref{NLS} in the Wiener algebra.  We begin with the well-posedness result.  We will construct a solution to the Duhamel formula for \eqref{NLS}, i.e.
\[
u(t)=e^{i(t+T)\Delta/2}u_0 -i \int_{-T}^t e^{i(t-s)\Delta/2}\alpha [|u|^2 u](s)\,ds.
\]

\begin{proposition}[Well-posedness in $\F L^1$]\label{WP} Let $d\in\{2,3\}$, $\alpha\in\F L^1$, and $T>0$. There exists $\delta_0=\delta_0(T,\|\hat\alpha\|_{L^1})>0$ such that for any $u_0\in\F L^1(\R^d)$ with
\[
\|\hat u_0\|_{L^1}<\delta_0, 
\]
there exists a unique solution $u\in L_t^\infty([-T,T];\F L^1(\R^d))$ to \eqref{NLS} satisfying
\begin{equation}\label{u-bds}
\|\hat u\|_{L_t^\infty L^1([-T,T]\times\R^d)}\leq 2\|\hat u_0\|_{L^1}. 
\end{equation}
\end{proposition}

\begin{proof} We fix $T>0$ and $\alpha\in \F L^1$.  Let $\delta_0>0$ be a small parameter to be determined below and let $u_0\in\F L^1$ satisfy
\[
\|\hat u_0\|_{L^1}<\delta_0.
\]
We define the complete metric space $(B,d)$ via
\[
B=\bigl\{u\in L_t^\infty([-T,T];\F L^1(\R^d)): \|\hat u\|_{L_t^\infty L^1([-T,T]\times\R^d)} \leq 2\|\hat u_0\|_{L^1}\bigr\}
\]
and
\[
d(u,v)=\|\hat u-\hat v\|_{L_t^\infty L^1([-T,T]\times\R^d)}. 
\]
We further define
\begin{equation}\label{Duhamel}
\Psi[u](t):=e^{i(t+T)\Delta/2}u_0 -i \int_{-T}^{t} e^{i(t-s)\Delta/2}\bigl[\alpha\,|u(s)|^2u(s)\bigr] \, ds.
\end{equation}
We will show that for $\delta_0$ sufficiently small, $\Psi$ is a contraction on $B$.

Let $u\in B$. For each $t\in[-T,T]$, we use \eqref{propagator} and the algebra property to estimate
\begin{align*}
\left\|\F[\Psi u](t) \right\|_{L^1}
&\leq \|\hat u_0\|_{L^1}
+ \left\|\int_{-T}^{t} e^{-i(t-s)|y|^2/2}\F\bigl[\alpha|u|^2 u\bigr] \, ds \right\|_{L^1} \\
&\leq \|\hat u_0\|_{L^1} + \int_{-T}^{t} \|\alpha\|_{\F L^1} \|u(s)\|_{\F L^1}^3\,ds \\
& \leq \|\hat u_0\|_{L^1} + 16T\delta_0^2\,\|\hat\alpha\|_{L^1} \|\hat u_0\|_{L^1} \\
& \leq 2\|\hat u_0\|_{L^1}
\end{align*}
for $\delta_0=\delta_0(T,\|\hat\alpha\|_{L^1})$ sufficiently small.  Taking the supremum over $t\in[-T,T]$, we obtain that $\Psi:B\to B$. 

Next, we let $u,v\in B$ and estimate as above to obtain
\[
\|\F[\Psi u](t) - \F[\Psi v](t)\|_{L^1}
 \leq \int_{-T}^t \|\hat \alpha\|_{L^1} \||u|^2 u - |v|^2 v\|_{\hat L^1} \,ds.
\]
We now observe that
\begin{equation}\label{algebraic}
|u|^2 u - |v|^2 v = (u-v)[|v|^2 + \bar v u] + \overline{(u-v)}u^2.
\end{equation}
Thus, by the algebra property, we obtain
\begin{align*}
\|\F[\Psi u](t) - \F[\Psi v](t)\|_{L^1} & \leq 24T\delta_0^2\,\|\hat \alpha\|_{L^1} \|\hat u-\hat v\|_{L_t^\infty L^1([-T,T]\times\R^d)} \\
& \leq \tfrac12 \|\hat u-\hat v\|_{L_t^\infty L^1([-T,T]\times\R^d)}
\end{align*}
for $\delta_0=\delta_0(T,\|\hat\alpha\|_{L^1})$ small enough. Taking the supremum over $t\in[-T,T]$, we deduce that $\Psi$ is a contraction. 

Applying the Banach fixed point theorem, we deduce that for $\delta_0$ sufficiently small, $\Psi$ has a unique fixed point $u\in B$, yielding the desired solution to \eqref{NLS} satisfying \eqref{u-bds}. 
\end{proof}

We next establish a stability result for \eqref{NLS} in $\F L^1$.

\begin{proposition}[Stability in $\F L^1$]\label{Stability} Let $d\in\{2,3\}$, $\alpha\in \F L^1$, and $T>0$.  Let $\delta_0>0$ be as in Proposition~\ref{WP}.

If $v:[-T,T]\times\R^d\to\C$ satisfies
\begin{align}\label{StaCon1}
\|\hat v(-T,x)\|_{ L^1(\R^d)}&<\delta_0, \\
\label{StaCon1.5}
\|\hat v\|_{L_t^\infty L^1(\R^d)}&\lesssim \delta_0,
\end{align}
and
\begin{equation}
\label{StaCon2}
\biggl\| \int_{-T}^t e^{i(t-s)\Delta/2}[(i\partial_t+\tfrac12\Delta) v-\alpha|v|^2 v](s)\,ds\biggr\|_{L_t^\infty \F L^1([-T,T]\times\R^d)} = \delta, \end{equation}
then the solution $u$ to \eqref{NLS} with $u|_{t=-T}=v|_{t=-T}$ exists on $[-T,T]$ and satisfies
\begin{equation}\label{close}
\|\hat u-\hat v\|_{L_t^\infty L^1([-T,T]\times\R^d)} \lesssim \delta. 
\end{equation}
\end{proposition}

\begin{proof} We fix $T>0$ and $\alpha\in\F L^1$.  Assume that $v:[-T,T]\times\R^d\to\C$ satisfies \eqref{StaCon1}--\eqref{StaCon2}.  We let $u:[-T,T]\times\R^d\to\C$ be the solution to \eqref{NLS} with $u|_{t=-T}=v|_{t=-T}$ satisfying \eqref{u-bds}, whose existence is guaranteed by Proposition~\ref{WP} and \eqref{StaCon1}.

We now observe that the difference $u-v$ satisfies the Duhamel formula
\[
u(t)-v(t) = -i\int_{-T}^t e^{i(t-s)\Delta/2}\bigl( \alpha[|u|^2 u - |v|^2 v](s)-E(s)\bigr)\,ds, 
\]
where
\[
E(t):=(i\partial_t + \tfrac12\Delta)v - \alpha |v|^2 v. 
\]
We now estimate as we did in the proof of Proposition~\ref{WP}.  Using \eqref{algebraic}, \eqref{StaCon2}, \eqref{StaCon1.5}, and \eqref{u-bds}, we have
\begin{align*}
\|\hat u(t)-\hat v(t)\|_{L^1} & \leq CT\|\hat\alpha\|_{L^1}\|\hat u-\hat v\|_{L_t^\infty L^1}(\|\hat u\|_{L_t^\infty L^1}^2+\|\hat v\|_{L_t^\infty L^1}^2) + \delta \\
& \leq CT\delta_0^2 \|\hat\alpha\|_{L^1}\, \|\hat u-\hat v\|_{L_t^\infty L^1} + \delta
\end{align*}
for some $C>0$, where all space-time norms are over $[-T,T]\times\R^d$.  Taking the supremum over $t\in[-T,T]$ and choosing $\delta_0$ smaller if necessary (so that $CT\delta_0^2\|\hat\alpha\|_{L^1}<\tfrac12$, say), we deduce that
\[
\|\hat u - \hat v\|_{L_t^\infty L^1([-T,T]\times\R^d)}\leq \tfrac12 \|\hat u-\hat v\|_{L_t^\infty L^1([-T,T]\times\R^d)}+\delta,
\]
which yields \eqref{close}, as desired.\end{proof}

\section{Construction of an approximate solution}\label{S:approx}

In this section we construct an approximate solution to \eqref{NLS} and prove the estimates needed to apply the stability result, Proposition~\ref{Stability}.

\begin{proposition}[Approximate solution]\label{ApproxSol} Fix $T>0$, $0<\eps\ll 1$, and $\xi\in\R^d$ with $|\xi|=1$.  Let $a_0\in \F L^1$ and define 
\begin{equation}\label{TESol}
a(t,x):=a_0(x-\tfrac{t\xi}{\eps})\exp\biggl\{-i\eps^{2p-2}|a_0(x-\tfrac{t\xi}{\eps})|^2\int_{0}^t \alpha\bigl[\tfrac{x}{\eps}-\tfrac{(t-s)\xi}{\eps^2} -T\xi \bigr]\,ds \biggr\}
\end{equation}
and
\begin{equation}\label{Ansatz}
v(t,x) :=\eps^{p}a(\eps^2 \left[t+T\right],\eps \left[x+T\xi\right])\exp\{i [x\cdot \xi -\tfrac12t]\}.
\end{equation}
Then we have the following:
\begin{itemize}
\item The function $v$ satisfies the initial condition
\begin{equation}\label{v-IC}
v(-T,x) = u_0(x):= \eps^p a_0(\eps[x+T\xi])\exp\{i[x\cdot\xi+\tfrac12T]\}.
\end{equation}
\item We have the following identity:
\begin{equation}\label{ECon}
\begin{aligned}
(i\partial_t & + \tfrac12\Delta)v - \alpha |v|^2 v \\
& = \tfrac12 \eps^{p+2}(\Delta a)(\eps^2 \left[t+T\right],\eps \left[x+T\xi\right])\exp\{i[x\cdot\xi-\tfrac12t]\}.
\end{aligned}
\end{equation}
\item The function a satisfies the following estimates: 
\begin{equation}\label{v-bds}
\|a\|_{L_t^\infty \F L^1}\lesssim_T 1.
\end{equation}
and
\begin{equation}\label{error-estimate}
\|\Delta a\|_{L_t^\infty \F L^1} \lesssim_{\sigma,T} \max\{1,\eps^{2p-4-\sigma}\}
\end{equation}
for all small $\sigma>0$.
\end{itemize}
\end{proposition}

\begin{proof} We begin by deriving the form of $a$ appearing in \eqref{TESol}, the initial condition \eqref{v-IC}, and the identity \eqref{ECon}.  We let 
\[
\Phi=\Phi(t,x)=x\cdot \xi -\tfrac12t
\]
and look for an approximate solution to \eqref{NLS} of the form 
\begin{equation*}
v(t,x) :=\eps^{p}a(\eps^2 \left[t+T\right],\eps \left[x+T\xi\right])e^{i \Phi}.
\end{equation*}
Noting that $(i\partial_t+\tfrac12\Delta)e^{i\Phi}=0$ and writing 
\[
(t',x') = (\eps^2 \left[t+T\right],\eps \left[x+T\xi\right]),
\]
direct computation yields
\begin{align*}
(i\partial_t& + \tfrac12\Delta)v - \alpha |v|^2 v\\
&= \eps^{p+2} e^{i\Phi} \bigl\{ i\partial_t a +\tfrac{i\xi}{\eps} \cdot \nabla a - \eps^{2p-2}\alpha(x)|a|^2a + \tfrac12\Delta a\bigr\},
\end{align*}
where $a$ and its derivatives are evaluated at $(t',x')$.
To obtain the identity \eqref{ECon} and the initial condition \eqref{v-IC}, we should therefore choose $a$ to satisfy
\[
\begin{cases}
&i\partial_t a +\frac{i\xi}{\eps} \cdot \nabla a- \eps^{2p-2}\alpha(\tfrac{x}{\eps}-T\xi)|a|^2a = 0, \\
& a|_{t=0}=a_0.
\end{cases}
\]
We solve this equation via the method of characteristics.  Setting $x(t)=x_0+\tfrac{t\xi}{\eps}$ for some $x_0\in\R^d$ and $b(t)=a(t,x(t))$, we find that the equation along the characteristic is given by
\[
\dot b = -i\eps^{2p-2}\alpha(\tfrac{1}{\eps}x(t)-T\xi)|b|^2 b. 
\]
Observing that
\[
\tfrac{d}{dt}|b|^2 = 2\Re[-i\eps^{2p-2}\alpha(\tfrac{1}{\eps}x(t)-T\xi)|b|^4] = 0,
\]
we obtain a linear ODE for $b$, which we solve via
\[
b(t) = a_0(x_0)\exp\biggl\{-i\eps^{2p-2}|a_0(x_0)|^2\int_0^t \alpha(\tfrac{1}{\eps}x(s)-T\xi)\,ds\biggr\}.
\]
Finally, observing that for given $x\in\R^d$, we have $x(t)=x$ for $x_0=x-\tfrac{t\xi}{\eps}$, we obtain the formula \eqref{TESol} for $a$. 

It remains to prove the estimates appearing in \eqref{v-bds} and \eqref{error-estimate}.  

We first consider \eqref{v-bds}. By \eqref{practical}, it suffices estimate the $H^{s}$-norm, where $s=\tfrac{d}{2}+\sigma$ for some small $\sigma>0$. To do this, we will estimate the $H^1$- and $H^2$-norms separately and then interpolate.  Applying the product and chain rules, we see that potentially dangerous terms arise when derivatives land on the coefficient $\alpha$, as each derivative produces a factor of $\eps^{-1}$.  However, as $\alpha$ appears in the phase, the first derivative acting on $\alpha$ necessarily produces a factor of $\eps^{2p-2}$, as well.  In addition, if any derivative lands on $\alpha$ then the resulting function is automatically localized to a ball of radius $\eps T$, which contributes a factor of $\eps^{\frac{d}{2}}$ in the evaluation of the $L^2$-norm.  Thus we find that
\[
\|a\|_{H^1} \lesssim \max\{1, \eps^{2p-3+\frac{d}{2}}\} \qtq{and} \|a\|_{H^2}\lesssim \max\{1, \eps^{2p-4+\frac{d}{2}}\},
\]
so that (recalling $s=\tfrac{d}{2}+\sigma\in(1,2)$)
\begin{equation}\label{v-bds4}
\|a\|_{\F L^1} \lesssim_{\sigma} \|a\|_{H^s} \lesssim_\sigma \max\{1,\eps^{2p-2-\sigma}\}.
\end{equation}
Recalling that $p>1$ and choosing $\sigma=\sigma(p)$ sufficiently small, we obtain \eqref{v-bds}. 

The argument for \eqref{error-estimate} is essentially the same.  In particular, we need to estimate the $H^s$-norm of $\Delta a$, where $s=\tfrac{d}{2}+\sigma$ for some small $\sigma>0$.  The main difference is that now two additional derivatives could act on $\alpha$, producing two additional powers of $\eps^{-1}$.  In particular, the same argument that led to \eqref{v-bds4} now yields \eqref{error-estimate}. 
\end{proof}

\section{Proof of the main theorem}\label{S:proof}

We turn now to the proof of Theorem~\ref{T}.

\begin{proof}[Proof of Theorem~\ref{T}] We let $d,p,a_0,\alpha,T$, and $\xi$ be as in the statement of the theorem.  Given $\eps>0$, we define $u_0$ as in \eqref{initial-data}. 

We define the approximate solution $v$ satisfying the initial condition \eqref{v-IC} given in Proposition~\ref{ApproxSol}. As
\[
\|u_0\|_{\F L^1} = \eps^p \|a_0\|_{\F L^1},
\]
Proposition~\ref{WP} guarantees that the true solution $u$ to \eqref{NLS} with $u|_{t=-T}=u_0$ exists on $[-T,T]$ and satisfies \eqref{u-bds} provided $\eps$ is chosen sufficiently small.  We wish to use Proposition~\ref{Stability} to prove that $u$ and $v$ remain close in the $\F L^1$-norm on the interval $[-T,T]$.  To this end, we first observe that by scaling and translation invariance of $L^1$, we have
\[
\|v\|_{L_t^\infty \F L^1}\lesssim \eps^p \|a\|_{L_t^\infty \F L^1} \lesssim \eps^p,
\]
so that \eqref{StaCon1.5} holds for $\eps$ sufficiently small. Next, we use \eqref{error-estimate} to estimate
\begin{align*}
\biggl\| \int_{-T}^t&  e^{i(t-s)\Delta/2}[(i\partial_t+\tfrac12\Delta) v-\alpha|v|^2 v](s)\,ds\biggr\|_{L_t^\infty \F L^1([-T,T]\times\R^d)} \\
& \lesssim_T \eps^{p+2}\|\Delta a\|_{L_t^\infty \F L^1} \lesssim_{\sigma,T} \max\{\eps^{p+2},\eps^{3p-2-\sigma}\}
\end{align*}
for all $\sigma$ sufficiently small.  Thus \eqref{StaCon2} also holds provided $\eps$ is chosen sufficiently small. 

Applying Proposition~\ref{Stability}, we deduce that
\[
\|u-v\|_{L_t^\infty \F L^1([-T,T]\times\R^d)} \lesssim_\sigma \max\{\eps^{p+2},\eps^{3p-2-\sigma}\},
\]
which completes the proof of Theorem~\ref{T}.\end{proof}




\end{document}